\pdfoutput=1
\documentclass[reqno]{amsart}

\usepackage{enumerate}
\usepackage{tabto}
\usepackage[mathscr]{euscript}
\usepackage{xcolor}
\usepackage{layout}
\usepackage{fancyhdr}
\usepackage{array}
\usepackage{amsfonts}
\usepackage{amsmath}
\usepackage{amssymb}
\usepackage{mathtools}
\usepackage{graphicx}
\usepackage{bm}
\usepackage{enumitem}
\usepackage{caption} 
\usepackage{color}
\usepackage{csquotes}
\usepackage{bookmark}
\usepackage{float}
\usepackage{multirow}
\usepackage[square,numbers,sort&compress]{natbib}
\usepackage{hyperref}
\hypersetup{colorlinks=true,linkcolor=blue,citecolor=red}
\allowdisplaybreaks

\def\Xint#1{\mathchoice
{\XXint\displaystyle\textstyle{#1}}%
{\XXint\textstyle\scriptstyle{#1}}%
{\XXint\scriptstyle\scriptscriptstyle{#1}}%
{\XXint\scriptscriptstyle\scriptscriptstyle{#1}}%
\!\int}
\def\XXint#1#2#3{{\setbox0=\hbox{$#1{#2#3}{\int}$ }
\vcenter{\hbox{$#2#3$ }}\kern-.6\wd0}}

\def\dashint{\Xint-}

\newtheorem{theorem}{Theorem}[section]
\newtheorem{lemma}[theorem]{Lemma}
\newtheorem{corollary}[theorem]{Corollary}

\newtheorem{remark}[theorem]{Remark}
\theoremstyle{definition}
\newtheorem{definition}[theorem]{Definition}

\numberwithin{equation}{section}

\newcommand{ \mr }{ \mathbb{R} }

\newcommand{\iints}[1]{{\int\hspace{-0.28cm}\int_{#1}}}
\newcommand{\iintss}{{\int\hspace{-0.28cm}\int}}
\newcommand{ \miints }{{\iintss\hspace{-0.56cm} -\hspace{-0.15cm}-}}
\newcommand{\miint}[1]{{\miints_{\hspace{-0.13cm}#1}}}

\begin{document}
\title[Local boundedness]{Local boundedness for solutions to parabolic double phase problems}

\author{Bogi Kim}\address{Department of Mathematics, Kyungpook National University, Daegu, 41566, Republic of Korea} \email{rlaqhrl4@knu.ac.kr} \author{Jehan Oh}\address{Department of Mathematics, Kyungpook National University, Daegu, 41566, Republic of Korea} \email{jehan.oh@knu.ac.kr}

\date{\today.}
\thanks{This work is supported by National Research Foundation of Korea (NRF) grant funded by the Korea government [Grant Nos. RS-2023-00217116, RS-2025-00555316, RS-2025-25415411, and RS-2025-25426375].}

\begin{abstract}
In this paper, we investigate the local boundedness of weak solutions to parabolic double phase equations of type
$$
u_t-\operatorname{div}(|Du|^{p-2}Du+a(x,t)|Du|^{q-2}Du)=0\quad \text{in }\Omega_T\coloneq \Omega\times (0,T),
$$
where $0\leq a(\cdot)\in L^\infty(\Omega_T)$ and $\frac{2n}{n+2}<p<q<p\left(\frac{n+2}{n}\right)$, so that both the degenerate range $p\geq 2$ and the singular range $p<2$ are covered.
To this end, we derive the Caccioppoli inequality and recall a parabolic embedding theorem, which are then utilized in an iteration method.
\end{abstract}
\maketitle

\section{\bf Introduction}\label{section 1}
In this paper, we investigate the local boundedness of weak solutions to homogeneous parabolic double phase problems with the model equation
\begin{equation}\label{main model equation}
        u_t-\operatorname{div} (|Du|^{p-2}Du + a(x,t)|Du|^{q-2}Du)=0 \qquad \text{in } \Omega_T\coloneq \Omega\times (0,T),
\end{equation}
where $n\geq 2$, $1<p<q$, $T>0$, $\Omega$ is a bounded open set in $\mr^n$ and $a(\cdot)$ is non-negative. To obtain the local boundedness, we assume that 
\begin{equation}\label{cond : basic condition of p and q}
a(\cdot)\in L^{\infty}(\Omega_T)\quad \text{and}\quad \frac{2n}{n+2}< p<q< p\left(\frac{n+2}{n}\right)\eqcolon \tilde{p}.
\end{equation}
Observe that $\tilde{p}>2$. For $2\leq p$, this condition is weaker than the gap bound condition
\begin{equation}\label{cond : second condition of p and q}
a(\cdot)\in C^{\alpha,\frac{\alpha}{2}}(\Omega_T)\quad \text{and}\quad q\leq p+\frac{2\alpha}{n+2},
\end{equation}
where $a(\cdot)\in C^{\alpha,\frac{\alpha}{2}}(\Omega_T)$ means that $a(\cdot)\in L^\infty (\Omega_T)$ and there exists a H\"{o}lder constant $[a]_\alpha\coloneq [a]_{\alpha,\frac{\alpha}{2};\Omega_T}>0$ such that
$$
|a(x_1,t_1)-a(x_2,t_2)|\leq [a]_\alpha \max\left\{|x_1-x_2|^\alpha,|t_1-t_2|^\frac{\alpha}{2}\right\}
$$
for all $x_1,\,x_2\in\Omega$ and $t_1,\,t_2\in(0,T)$. The relationship between condition $(1.2)_2$ and condition $(1.3)_2$ is as follows: since $p>\frac{n}{n+2}$ under \eqref{cond : basic condition of p and q},
\begin{equation}\label{eq : relationship between 2/n+2 and p/n}
\frac{2}{n+2}< \frac{2p}{n}.
\end{equation}
Also, since 
$$
\frac{p(n+2)-2n}{2(n+2)}\leq \frac{p}{n+2}<\frac{2p}{n}\quad \text{for }p\leq 2,
$$
when $\frac{2n}{n+2}<p\leq 2$, the condition \eqref{cond : basic condition of p and q} is weaker than the gap bound condition in \cite{Wontae2024}
$$
a(\cdot)\in C^{\alpha,\frac{\alpha}{2}}(\Omega_T)\quad \text{and}\quad q\leq p+\frac{\alpha(p(n+2)-2n)}{2(n+2)}.
$$
Under the condition \eqref{cond : second condition of p and q}, Kim-Kinnunen-S\"{a}rki\"{o} \cite{Wontae2023a} have studied the existence and uniqueness theory for weak solutions (see also \cite{Chlebicks2019,Singer2016}) and Kim-Kinnunen-Moring \cite{2023_Gradient_Higher_Integrability_for_Degenerate_Parabolic_Double-Phase_Systems} have obtained the gradient higher integrability results of the solution for the degenerate parabolic double phase problems. Moreover, Kim-Oh \cite{kim2025boundedsolutionsinterpolativegap} and Kim-Moring-S\"{a}rki\"{o} \cite{Wontae2025} have investigated the higher integrability results and the H\"{o}lder continuity for weak solutions under the assumption that
$$
u\in L^\infty(\Omega_T),\quad a(\cdot)\in C^{\alpha,\frac{\alpha}{2}}(\Omega_T)\quad \text{and}\quad q\leq p +\alpha.
$$ 
More generally, Kim-Oh \cite{kim2025boundedsolutionsinterpolativegap} have extended the higher integrability results to the setting, where 
$$
u\in C(0,T;L^s(\Omega)),\quad a(\cdot)\in C^{\alpha,\frac{\alpha}{2}}(\Omega_T)\quad\text{and}\quad q\leq p +\frac{s\alpha}{n+s}.
$$
In addition, \cite{Wontae2023b,Wontae2024,kim2026interpolativerefinementgapbound} provide further results regarding singular parabolic double phase problems, while various results related to parabolic double phase problems can be found in \cite{Buryachenko2022,Kim2025,Kim2024,Sen2025,kim2026absencelavrentievphenomenondegenerate}.

Now, we introduce the main equations. The main equations under consideration are of the form 
\begin{equation}\label{eq : the main equation}
    u_t-\operatorname{div}\mathcal{A}(z,Du)=0\qquad \text{in }\Omega_T.
\end{equation}
Here, $\mathcal{A}:\Omega_T \times \mr^{n}\rightarrow \mr^{n}$ is a Carath\'{e}odory vector field satisfying the following: there exist constants $0<\nu\leq L <\infty$ such that
\begin{equation}    \label{cond : double phase bounded condition of integrand}
    \mathcal{A}(z,\xi)\cdot \xi\geq \nu (|\xi|^p+a(z)|\xi|^q)\quad \text{and}\quad |\mathcal{A}(z,\xi)|\leq L(|\xi|^{p-1}+a(z)|\xi|^{q-1})
\end{equation}
for all $z\in\Omega_T$ and $\xi\in \mr^n$. For simplicity, we denote $H(z,\varkappa)\coloneq  \varkappa^p +a(z)\varkappa^q$ for $\varkappa\geq 0$ and $z\in\Omega_T$. The definition of weak solutions to \eqref{eq : the main equation} is as follows:
\begin{definition}
    A function $u:\Omega_T \rightarrow \mr$ with
    $$
    u\in C(0,T;L^2(\Omega))\cap L^{q}(0,T;W^{1,q}(\Omega))
    $$
    and
    $$
    \iints{\Omega_T} H(z,|Du|)\, dz <\infty
    $$
    is a weak solution to \eqref{eq : the main equation}, if
    $$
    \begin{aligned}
        \iints{\Omega_T} (-u\cdot\varphi_t +\mathcal{A}(z,Du)\cdot D\varphi)\, dz=0
    \end{aligned}
    $$
    for every test function $\varphi\in C_0^\infty (\Omega_T)$.
\end{definition}
\begin{remark}\label{rmk : natural energy space}
    A more natural function space assumption in the definition above would be
    $$
    u\in C(0,T;L^2(\Omega))\cap L^1(0,T;W^{1,1}(\Omega))\quad\text{with}\quad \iints{\Omega_T} H(z,|Du|)\, dz<\infty,
    $$
    without requiring $|Du|\in L^q(\Omega_T)$. As in \cite[Remark 2.2]{2023_Gradient_Higher_Integrability_for_Degenerate_Parabolic_Double-Phase_Systems}, the assumption $|Du|\in L^q(\Omega_T)$ is applied only in the proof of the Caccioppoli inequality in Lemma \ref{lem : Caccioppoli inequality}: when testing the Steklov averaged equation \eqref{eq : Steklov average of main problems}, there appear cross terms of the form $a(x,t)|Du(x,t)|^{q-1}|Du(x,s)|$ with $t\neq s$, in which the coefficient and the gradient are evaluated at different times after Young's inequality, and such terms are not controlled by $H(z,|Du|)\in L^1(\Omega_T)$ alone. This difficulty does not occur in the elliptic case, since no mollification in time is needed there. On the other hand, the remaining part of the proof of Theorem \ref{thm : main theorem} only uses quantities controlled by the natural energy. In this respect, we point out the following:
    \begin{enumerate}[label=\upshape{(\roman*)}]
        \item If $a(x,t)=a(x)$ is independent of time, then the natural assumption above is sufficient. Indeed, the cross terms are then estimated by $a(x)|Du(x,t)|^q+a(x)|Du(x,s)|^q$, which is controlled by the natural energy, while the term $a(x)((u-k)_\pm|D\eta|)^q$ is locally integrable since $u\in L^{\tilde{p}}_{\mathrm{loc}}(\Omega_T)$ by Lemma \ref{lem : a parabolic embedding result} and $q<\tilde{p}$.
        \item Under the stronger assumption \eqref{cond : second condition of p and q}, we expect that the energy estimate can be obtained in the natural function space by adapting the parabolic Lipschitz truncation method developed in \cite{Wontae2023a}, where weak solutions are considered in the natural function space; see also \cite[Remark 2.2]{2023_Gradient_Higher_Integrability_for_Degenerate_Parabolic_Double-Phase_Systems}. In this case Theorem \ref{thm : main theorem} would remain valid under the natural assumption, provided that \eqref{cond : second condition of p and q} and $2\leq p$ hold. We do not pursue this here.
        \item For a time-dependent coefficient $a(\cdot)$ that is merely bounded and measurable, as in \eqref{cond : basic condition of p and q}, deriving the energy estimate under the natural assumption seems to be an open problem, since the Lipschitz truncation technique in \cite{Wontae2023a} relies on the intrinsic geometry and the continuity of $a(\cdot)$.
    \end{enumerate}
\end{remark}
We write parabolic cylinders as
$$
\mathcal{Q}_{R,\ell}(z_0)=B_R(x_0)\times \mathcal{I}_\ell(t_0),\quad R,\,\ell>0,
$$
where
$$
\mathcal{I}_\ell(t_0) \coloneq (t_0-\ell,t_0).
$$

The main theorem concerns the local boundedness of weak solutions.
\begin{theorem}\label{thm : main theorem}
    Let $u$ be a weak solution to \eqref{eq : the main equation} in $\Omega_T$ under the assumptions \eqref{cond : double phase bounded condition of integrand} and \eqref{cond : basic condition of p and q}. Then $u$ is locally bounded in $\Omega_T$. Moreover, there exists a positive constant $c=c(n,p,q,\nu,L,\|a\|_{L^\infty(\Omega_T)})>1$ such that for $\sigma\in (0,1)$ and $\rho\in(0,1)$ with $\mathcal{Q}_{\rho,\rho^{\tilde{p}}}(z_0)\subset\subset\Omega_T$,
    $$
    \begin{aligned}
        \sup_{\mathcal{Q}_{\frac{\sigma\rho}{2},\sigma^2 \left(\frac{\rho}{2}\right)^{\tilde{p}}}(z_0)} |u| &\leq 1+c (1-\sigma)^{-\frac{d}{\tilde{p}-d}}\left(\miint{\mathcal{Q}_{\rho,\rho^{\tilde{p}}}(z_0)}\left[|Du|^p+\frac{|u|^p}{\rho^p}\right]\, dz\right)^{\frac{\vartheta}{(\tilde{p}-d)(1+\vartheta)}}\\
        &\qquad \times\left(\sup_{t\in\mathcal{I}_{\rho^{\tilde{p}}}(t_0)}\dashint_{B_\rho(x_0)}\frac{|u|^2}{\rho^2}\, dx\right)^{\frac{\vartheta}{(\tilde{p}-d)(1+\vartheta)}\frac{p}{n}},
    \end{aligned}
    $$
    where $\vartheta\coloneq \frac{p}{n}$ and $d\coloneq \max\{q,2\}$.
\end{theorem}
\begin{remark}
    Note that $d=q$ whenever $q\geq 2$; in particular, this is the case in the degenerate range $p\geq 2$. The case $d=2>q$ occurs only in the singular range $\frac{2n}{n+2}<p<q<2$, in which the $L^2$-term coming from the time derivative dominates the right-hand side of the Caccioppoli inequality for large levels $|k|$. In the remaining range $1<p\leq \frac{2n}{n+2}$, we have $\tilde{p}\leq 2$ and the H\"{o}lder-type estimate \eqref{eq : estimate of level set on +} with $s=2$, which requires $\tilde{p}>2$, breaks down; it would be interesting to investigate whether an additional integrability assumption on $u$ allows one to treat this range as well, cf. \cite{Singer2016NoDEA}.
\end{remark}
\begin{remark}\label{rmk : strict gap}
    The strict inequality $q<\tilde{p}$ in \eqref{cond : basic condition of p and q} is intrinsic to the De Giorgi iteration employed in this paper. Indeed, the smallness for large levels in the iteration comes from the factor $|k|^{d-\tilde{p}}$ in \eqref{eq : final recursive inequality} below, and its exponent vanishes in the borderline case $q=\tilde{p}$. For variational solutions, the local boundedness in the borderline case $q=\tilde{p}$ has been established in \cite{Singer2016NoDEA} by a different method. It would be an interesting question whether the borderline case can be reached for weak solutions to \eqref{eq : the main equation}.
\end{remark}

Let us compare Theorem \ref{thm : main theorem} with the result of Singer \cite{Singer2016NoDEA}, where the local boundedness of variational solutions and parabolic minimizers was established for integrands $f(x,\xi)$ that are convex in $\xi$, independent of time, and satisfy the non-standard $p,q$-growth condition
$$
\nu|\xi|^p\leq f(x,\xi)\leq L(1+|\xi|^q)
$$
for $0<\nu\leq L$ and $1<p\leq q\leq p\left(\frac{n+2}{n}\right)=\tilde{p}$. Although the model equation \eqref{main model equation} is formally the gradient flow of a functional with $p,q$-growth, the two settings differ in several respects. First, \cite{Singer2016NoDEA} concerns variational solutions, which are defined through a minimality property of the functional; the Caccioppoli type estimates there are obtained by choosing suitable competitors in the minimality inequality, for which the convexity and the time-independence of the integrand are essential. In contrast, we consider weak solutions to \eqref{eq : the main equation}, where the Carath\'{e}odory vector field $\mathcal{A}$ is not assumed to have any variational structure, but only satisfies the double phase growth and coercivity condition \eqref{cond : double phase bounded condition of integrand} with a coefficient $a(x,t)\geq 0$ that depends also on time and is merely bounded and measurable. In particular, even for the model equation \eqref{main model equation}, the time-dependence of $a(\cdot)$ excludes the framework of \cite{Singer2016NoDEA}. Moreover, weak solutions and variational solutions are in general not known to coincide under $p,q$-growth conditions; the identification requires considerably stronger assumptions such as \eqref{cond : second condition of p and q}, see \cite{Wontae2023a}. Hence neither result contains the other. Second, the parameter ranges differ at the borderline: while \cite{Singer2016NoDEA} includes $q=\tilde{p}$, our range \eqref{cond : basic condition of p and q} is strict. The reason is that the comparison arguments based on minimality are not available for weak solutions; instead, we test the Steklov averaged equation directly and run a De Giorgi iteration, whose smallness for large levels degenerates as $q\nearrow\tilde{p}$, see Remark \ref{rmk : strict gap}. Finally, the double phase structure introduces an additional difficulty: since the coefficient $a(\cdot)$ may vanish on a large set and no continuity of $a(\cdot)$ is available under \eqref{cond : basic condition of p and q}, the Caccioppoli inequality of Lemma \ref{lem : Caccioppoli inequality} has to keep the inhomogeneous quantity $H(z,\cdot)$ on both sides of the estimate, so that it applies to both phases simultaneously.

The local boundedness is essential for obtaining the H\"{o}lder continuity. In fact, in the elliptic double phase problems of the form
$$
-\operatorname{div}(|Du|^{p-2}Du+a(x)|Du|^{q-2}Du)=0\quad \text{in }\Omega
$$
with $1<p\leq q$ and $0\leq a(\cdot)\in C^{\alpha}(\Omega)$ for some $\alpha\in (0,1]$ (originally introduced in \cite{Zhikov1986,Zhikov1993,Zhikov1995,Zhikov1997}),
local boundedness is established as an a priori condition for proving the H\"{o}lder continuity of weak solutions; see \cite{Ok2017,Colombo2015,Baroni2018}. Here, we assume 
$$
\frac{q}{p}\leq 1 +\frac{\alpha}{n}.
$$
This assumption is the sharp condition to obtain the regularity results of weak solutions for the elliptic double phase problems, see \cite{Esposito2004,Colombo2015}.

As in the elliptic case, the H\"{o}lder continuity results for the parabolic double phase problems can be derived from the local boundedness in Theorem \ref{thm : main theorem}. By combining this theorem, \eqref{eq : relationship between 2/n+2 and p/n} and \cite{Wontae2025}, we deduce the H\"{o}lder continuity of $u$ under the assumption \eqref{cond : second condition of p and q} together with $2\leq p$. More precisely, since \eqref{eq : relationship between 2/n+2 and p/n} implies
$$
\frac{2\alpha}{n+2}< \frac{2p}{n},
$$
the condition \eqref{cond : second condition of p and q} with $2\leq p$ leads to \eqref{cond : basic condition of p and q}, as $2\leq p$ gives $p>\frac{2n}{n+2}$. Consequently, we obtain the local boundedness of $u$ from Theorem \ref{thm : main theorem}. Furthermore, as 
$$
\frac{2\alpha}{n+2}< \alpha,
$$
we have the following corollary by \cite{Wontae2025}.
\begin{corollary}\label{cor : Holder continuity}
    Let $u$ be a weak solution to \eqref{eq : the main equation} in $\Omega_T$ under the assumptions \eqref{cond : double phase bounded condition of integrand}, \eqref{cond : second condition of p and q} and $2\leq p$. Then $u$ is locally H\"{o}lder continuous in $\Omega_T$.
\end{corollary}

The proof of Theorem \ref{thm : main theorem} is based on the methodology established in \cite[Chapter V]{1993_Degenerate_parabolic_equations_DiBenedetto} for the local boundedness of $p$-Laplace problems. The parabolic embedding \cite[Chapter I, Proposition 3.1]{1993_Degenerate_parabolic_equations_DiBenedetto}, recalled in Lemma \ref{lem : a parabolic embedding result}, is applied to the truncations $(u-k)_\pm$ under the assumption \eqref{cond : basic condition of p and q}, while the double phase structure requires the Caccioppoli inequality in Lemma \ref{lem : Caccioppoli inequality} to keep the inhomogeneous quantity $H(z,\cdot)$ on both sides of the estimate. In Section \ref{section 2}, we introduce the notations and prove the Caccioppoli inequality for weak solutions of the parabolic double phase equations. Lastly, using the iteration method, we prove the main theorem.

\section{\bf Preliminary}\label{section 2}
\subsection{Notations}
For a fixed point $z_0 \in \Omega_T$, we denote
\begin{equation}\label{def : definition of H with a fixed center z_0}
    H_{z_0}(\varkappa)\coloneq \varkappa^p+a(z_0)\varkappa^q \qquad \text{for } \varkappa\geq 0.
\end{equation}
Also, we write
$$
\mathcal{Q}_{R,\ell}(z_0)= B_R(x_0)\times \mathcal{I}_\ell (t_0),
$$
where
$$
B_\rho (x_0)=\{x\in \mr^n : |x-x_0|<\rho\}
$$
and
$$
\mathcal{I}_\ell(t_0)=(t_0-\ell,t_0).
$$
The integral average of $f\in L^1(\Omega_T)$ over a measurable set $E\subset \Omega_T$ with $0<|E|<\infty$ is denoted by
$$
f_E=\frac{1}{|E|}\iints{E} f\, dz=\miint{E} f\, dz.
$$
Also, the spatial integral average of $f\in L^1(\Omega_T)$ over an $n$-dimensional ball $B\subset \Omega$ is denoted by
$$
f_{B}(t)=\dashint_{B} f(x,t) \, dx.
$$

\subsection{Caccioppoli inequality}
For a fixed cylinder $\mathcal{Q}_{R,\ell}(z_0)\subset\subset \Omega_T$, we consider two cut-off functions $\eta\in C_0^\infty(B_R(x_0))$ and $\zeta\in C^\infty(\mr)$ non-decreasing with $\zeta(t)=0$ for $t\leq t_0-\ell$, satisfying $0\leq \eta\leq 1$ and $0\leq \zeta\leq 1$.
\begin{lemma}\label{lem : Caccioppoli inequality}
    Let $u$ be a weak solution to \eqref{eq : the main equation} in $\Omega_T$. Then there exists a positive constant $c=c(p,q,\nu,L)$ such that for every $\mathcal{Q}_{R, \ell}\left(z_0\right) \subset\subset \Omega_T$, with $R, \ell>0$ and any level $k$,
    $$
    \begin{aligned}
        &\sup_{t\in\mathcal{I}_{\ell}(t_0)}\dashint_{B_{R}(x_0)} \frac{(u-k)^2_\pm \eta^q\zeta^2}{\ell}\, dx + \miint{\mathcal{Q}_{R,\ell}(z_0)} \eta^qH(z,|D(u-k)_\pm|)\zeta^2\, dz\\
        &\qquad\quad\leq c\miint{\mathcal{Q}_{R,\ell}(z_0)} H(z, (u-k)_\pm|D\eta|)\zeta^2\, dz + c\miint{\mathcal{Q}_{R,\ell}(z_0)} (u-k)_\pm^2 \eta^q\zeta\partial_t\zeta\, dz.
    \end{aligned}
    $$
\end{lemma}
\begin{proof}
    Let $\mathcal{Q}_{R,\ell}(z_0)\subset\subset \Omega_T$. Let $t_*\in \mathcal{I}_\ell (t_0)$ and $\delta>0$. We define $\zeta_\delta$ by
    $$
    \zeta_\delta(t)\coloneq \begin{cases}
        1,\qquad &\text{if }t\in (-\infty,t_*-\delta),\\
        1-\frac{t-t_*+\delta}{\delta},\qquad &\text{if }t\in [t_*-\delta,t_*],\\
        0,\qquad &\text{if }t\in(t_*,\infty).
    \end{cases}
    $$

    Fix $h\in(0,T-t_0)$, where $t_0<T$ since $\mathcal{Q}_{R,\ell}(z_0)\subset\subset\Omega_T$. We consider \eqref{eq : the main equation} in terms of Steklov averages and obtain
    \begin{equation}\label{eq : Steklov average of main problems}
        \partial_t u_h -\operatorname{div}[\mathcal{A}(\cdot,Du)]_h = 0
    \end{equation}
    in $\mathcal{Q}_{R,\ell}(z_0)$. Note that
    $$
    (u_h-k)_{\pm}\eta^q\zeta^2\zeta_\delta\in W_0^{1,2}(\mathcal{I}_{\ell}(t_0);L^2(B_R(x_0)))\cap L^q(\mathcal{I}_{\ell}(t_0);W_0^{1,q}(B_R(x_0))).
    $$
    By applying $\varphi=\pm(u_h-k)_{\pm}\eta^q\zeta^2\zeta_\delta$ as a test function in \eqref{eq : Steklov average of main problems}, we have
    \begin{equation}\label{eq : estimate in Caccioppoli inequality}
        \begin{aligned}
            &\mathrm{I}+\mathrm{II}=\miint{\mathcal{Q}_{R,\ell}(z_0)} \partial_t u_h \varphi\, dz+\miint{\mathcal{Q}_{R,\ell}(z_0)} [\mathcal{A}(\cdot,Du)]_h \cdot D\varphi\, dz=0.
        \end{aligned}
    \end{equation}
    Note that $\mathrm{II}$ is finite under the assumption $|Du|\in L^q(\Omega_T)$, see \cite{2023_Gradient_Higher_Integrability_for_Degenerate_Parabolic_Double-Phase_Systems}; see also Remark \ref{rmk : natural energy space}.

    \textbf{Estimate of $\mathrm{I}$:} By integration by parts, we have
    $$
    \begin{aligned}
        \mathrm{I}&=\frac{1}{2}\miint{\mathcal{Q}_{R,\ell}(z_0)} \partial_t (u_h-k)_\pm^2 \eta^q\zeta^2\zeta_\delta\, dz\\
        &=-\miint{\mathcal{Q}_{R,\ell}(z_0)} (u_h-k)_\pm^2 \eta^q\zeta\zeta_\delta\partial_t\zeta\, dz\\
        &\qquad - \frac{1}{2}\miint{\mathcal{Q}_{R,\ell}(z_0)} (u_h-k)_\pm^2 \eta^q\zeta^2 \partial_t\zeta_\delta\, dz\\
        &\geq -\miint{\mathcal{Q}_{R,\ell}(z_0)} (u_h-k)_\pm^2 \eta^q\zeta\partial_t\zeta\, dz\\
        &\qquad + \frac{1}{2|\mathcal{Q}_{R,\ell}|}\dashint_{t_*-\delta}^{t_*}\int_{B_R(x_0)} (u_h-k)_\pm^2 \eta^q\zeta^2\,dx dt.
    \end{aligned}
    $$
    Thus, we get
    $$
    \begin{aligned}
        \lim_{h\rightarrow 0^+}\lim_{\delta\rightarrow 0^+} \mathrm{I} &\geq -\miint{\mathcal{Q}_{R,\ell}(z_0)} (u-k)_\pm^2 \eta^q\zeta\partial_t\zeta\, dz\\
        &\qquad + \frac{1}{2|\mathcal{Q}_{R,\ell}|} \int_{B_R(x_0)} (u(x,t_*)-k)_\pm^2 \eta^q\zeta^2(t_*)\,dx
    \end{aligned}
    $$

    \textbf{Estimate of $\mathrm{II}$:} Observe that
    $$
    \begin{aligned}
        \mathrm{II}&=\miint{\mathcal{Q}_{R,\ell}(z_0)} [\mathcal{A}(z,Du)]_h\cdot [\pm D(u_h-k)_\pm \eta^q\zeta^2\zeta_\delta\pm q (u_h-k)_\pm \eta^{q-1}D\eta\zeta^2\zeta_\delta]\, dz.
    \end{aligned}
    $$
    We first pass to the limit. Since $\zeta_\delta\rightarrow \chi_{(-\infty,t_*)}$ as $\delta\rightarrow 0^+$, and since $[\mathcal{A}(\cdot,Du)]_h\rightarrow \mathcal{A}(\cdot,Du)$ in $L^{\frac{q}{q-1}}(\mathcal{Q}_{R,\ell}(z_0))$ and $(u_h-k)_\pm\rightarrow (u-k)_\pm$ in $L^q(\mathcal{I}_\ell(t_0);W^{1,q}(B_R(x_0)))$ as $h\rightarrow 0^+$, we obtain
    $$
    \begin{aligned}
        \lim_{h\rightarrow 0^+}\lim_{\delta\rightarrow 0^+}\mathrm{II}&=\frac{1}{|\mathcal{Q}_{R,\ell}|}\int_{\mathcal{I}_\ell (t_0)\cap (-\infty,t_*)}\int_{B_{R}(x_0)} \mathcal{A}(z,Du)\cdot \big[\pm D(u-k)_\pm \eta^q\zeta^2\\
        &\qquad\qquad\qquad\qquad\qquad\qquad\qquad\qquad \pm q (u-k)_\pm \eta^{q-1}D\eta\zeta^2\big]\, dx dt.
    \end{aligned}
    $$
    Now we apply the structure condition \eqref{cond : double phase bounded condition of integrand}. Since $Du=D(u-k)_\pm$ on $\{\pm(u-k)>0\}$ and $(u-k)_\pm=0$ elsewhere, we have
    $$
    \pm\mathcal{A}(z,Du)\cdot D(u-k)_\pm=\mathcal{A}(z,Du)\cdot Du\,\chi_{\{\pm(u-k)>0\}}\geq \nu H(z,|D(u-k)_\pm|)
    $$
    and
    $$
    |\mathcal{A}(z,Du)|(u-k)_\pm\leq L\left(|D(u-k)_\pm|^{p-1}+a(z)|D(u-k)_\pm|^{q-1}\right)(u-k)_\pm.
    $$
    Applying Young's inequality to the second term together with $\eta^{q-p}\leq 1$, which follows from $0\leq \eta \leq 1$ and $p<q$, we get
    $$
    q|\mathcal{A}(z,Du)|(u-k)_\pm\eta^{q-1}|D\eta|\leq \frac{\nu}{2}\eta^qH(z,|D(u-k)_\pm|)+cH(z,(u-k)_\pm|D\eta|)
    $$
    for some $c=c(p,q,\nu,L)$. Therefore, we obtain
    $$
    \begin{aligned}
        \lim_{h\rightarrow 0^+}\lim_{\delta\rightarrow 0^+}\mathrm{II}&\geq \frac{\nu}{2|\mathcal{Q}_{R,\ell}|}\int_{\mathcal{I}_\ell (t_0)\cap (-\infty,t_*)}\int_{B_{R}(x_0)} \eta^qH(z,|D(u-k)_\pm|)\zeta^2\, dx dt\\
        &\qquad - c \miint{\mathcal{Q}_{R,\ell}(z_0)} H(z,(u-k)_\pm |D\eta|)\zeta^2\, dz.
    \end{aligned}
    $$
    
    Combining the above inequalities in \eqref{eq : estimate in Caccioppoli inequality}, we have
    $$
    \begin{aligned}
        &\frac{1}{|\mathcal{Q}_{R,\ell}|}\int_{B_R(x_0)} (u(x,t_*)-k)_\pm^2 \eta^q\zeta^2\,dx\\
        &\qquad + \frac{1}{|\mathcal{Q}_{R,\ell}|}\int_{\mathcal{I}_\ell (t_0)\cap (-\infty,t_*)}\int_{B_{R}(x_0)} \eta^qH(z,|D(u-k)_\pm|)\zeta^2\, dx dt\\  
        &\quad \leq c\miint{\mathcal{Q}_{R,\ell}(z_0)} H(z, (u-k)_\pm|D\eta|)\zeta^2\, dz + c \miint{\mathcal{Q}_{R,\ell}(z_0)} (u-k)_\pm^2 \eta^q\zeta\partial_t\zeta\, dz
    \end{aligned}
    $$
    for some $c=c(p,q,\nu,L)$. Since $t_*\in \mathcal{I}_\ell (t_0)$ is arbitrary, we have the conclusion.
\end{proof}
\section{\bf Local boundedness of weak solutions}
To prove Theorem \ref{thm : main theorem}, we need a parabolic embedding result as follows.
\begin{lemma}\cite[Chapter I, Proposition 3.1]{1993_Degenerate_parabolic_equations_DiBenedetto} \label{lem : a parabolic embedding result}
    There exists $c=c(n,p)>1$ such that for every $\mathcal{Q}_{R,\ell}(z_0)\subset\subset \Omega_T$ and for any $$f\in C(\mathcal{I}_{\ell}(t_0);L^2(B_R(x_0)))\cap L^p(\mathcal{I}_\ell (t_0);W_0^{1,p}(B_R(x_0))),$$ 
    we have
    $$
    \begin{aligned}
    \iints{\mathcal{Q}_{R,\ell}(z_0)} \left|f\right|^{\tilde{p}}\, dz&\leq c\left(\iints{\mathcal{Q}_{R,\ell}(z_0)}|Df|^p\, dz\right)\left(\sup_{t\in\mathcal{I}_\ell(t_0)}\int_{B_R(x_0)}\left|f\right|^2\,dx\right)^{\frac{p}{n}}.
    \end{aligned}
    $$
\end{lemma}
\begin{remark}
    Let $\mathcal{Q}_{R,\ell}(z_0)\eqcolon \mathcal{Q}\eqcolon B\times \mathcal{I}$, $\hat{\mathcal{Q}}=\hat{B}\times\hat{\mathcal{I}}$ with $\hat{B}\subset\subset B$ and $\hat{\mathcal{I}}\subset\mathcal{I}$, and let $f\in C(\mathcal{I};L^2(B))\cap L^p(\mathcal{I};W^{1,p}(B))$. Choose a piecewise smooth cutoff function $\iota=\iota(x)$ with compact support in $B$ satisfying $0\leq \iota\leq 1$ and $\iota\equiv 1$ in $\hat{B}$. Then $f\iota \in C(\mathcal{I};L^2(B))\cap L^p(\mathcal{I};W_0^{1,p}(B))$ and hence by Lemma \ref{lem : a parabolic embedding result}, we have
    \begin{equation*}
        \iints{\hat{\mathcal{Q}}} |f|^{\tilde{p}}\, dz\leq c\left(\iints{\mathcal{Q}}[|Df|^p + |f|^p|D\iota|^p]\, dz\right)\left(\sup_{t\in\mathcal{I}}\int_{B}\left|f\right|^2\,dx\right)^{\frac{p}{n}}.
    \end{equation*}
    Since $R^{\tilde{p}}=R^{p+\frac{2p}{n}}=(R^n)^{\frac{p}{n}}(R^2)^{\frac{p}{n}}$, dividing both sides by $R^{\tilde{p}}$ yields
    \begin{equation}\label{eq : a parabolic embedding inequality}
        \iints{\hat{\mathcal{Q}}} \left|\frac{f}{R}\right|^{\tilde{p}}\, dz\leq c\left(\iints{\mathcal{Q}}[|Df|^p+|f|^p|D\iota|^p]\, dz\right)\left(\sup_{t\in\mathcal{I}}\dashint_{B}\left|\frac{f}{R}\right|^2\,dx\right)^{\frac{p}{n}}.    
    \end{equation}
\end{remark}
\begin{proof}[Proof of Theorem \ref{thm : main theorem}]
    To prove this, we use an iteration method and so we introduce some notation. Fix $\sigma\in(0,1)$, $R,\, \ell\leq 1$ and write
    $$
    \begin{aligned}
        R_m&\coloneq \sigma R +\frac{1-\sigma}{2^m}R\\
        \ell_m&\coloneq \sigma \ell +\frac{1-\sigma}{2^m}\ell\\
        k_m^\pm&\coloneq k-\frac{k}{2^m},\quad k\in\mr^\pm
    \end{aligned}
    $$
    for $m=0,\,1,\,2,\cdots$, where $\mr^+\coloneq\{r\in\mr : r>0\}$ and $\mr^-\coloneq \{r\in\mr : r<0\}$. For each $m=0,\,1,\,2,\cdots$, we set
    $$
    B_m\times I_m = Q_m \coloneq \mathcal{Q}_{R_m,\ell_m^2}(z_0),\quad Q_0=\mathcal{Q}_{R,\ell^2}(z_0),\quad Q_\infty\coloneq \mathcal{Q}_{\sigma R, \sigma^2 \ell^2} (z_0)
    $$
    and
    $$
    \tilde{B}_m\times \tilde{I}_m=\tilde{Q}_m\coloneq \mathcal{Q}_{\tilde{R}_m,\tilde{\ell}_m^2}(z_0),
    $$
    where $\tilde{R}_m\coloneq \frac{R_m+R_{m+1}}{2}=\sigma R+\frac{3(1-\sigma)}{2^{m+2}}R$ and $\tilde{\ell}_m\coloneq \frac{\ell_m+\ell_{m+1}}{2}=\sigma\ell+\frac{3(1-\sigma)}{2^{m+2}}\ell$, so that $I_m=(t_0-\ell_m^2,t_0)$ and $\tilde{I}_m=(t_0-\tilde{\ell}_m^2,t_0)$. Then we observe that $Q_{m+1}\subset \tilde{Q}_m\subset Q_m$ for any $m=0,\,1,\, 2,\cdots$. If $k > 0$, $$k^+_{m+1}=k-\frac{k}{2^{m+1}}\geq k-\frac{k}{2^m}=k^+_m,$$ 
    but if $k< 0$, 
    $$k^-_{m+1}=k-\frac{k}{2^{m+1}}\leq k-\frac{k}{2^m}=k^-_m.$$
    Moreover, for $m=0,\,1,\,2,\cdots$ we set the level sets
    $$
    A^+_{m+1}\coloneq\{(x,t)\in Q_m : u(x,t)>k_{m+1}^+ \},
    $$
    $$
    A^-_{m+1}\coloneq\{(x,t)\in Q_m : u(x,t)<k_{m+1}^- \}.
    $$
    Consider test functions $\eta\in C_0^\infty(B_m)$ and $\zeta\in C^\infty(\mr)$ non-decreasing with $\zeta(t)=0$ for $t\leq t_0-\ell_m^2$ satisfying 
    $$
    \eta\equiv 1 \quad \text{in } \tilde{B}_m,\quad 0\leq \eta \leq 1 \quad\text{and}\quad |D\eta|\leq \frac{2}{R_m-\tilde{R}_m}
    $$
    and
    $$
    \zeta\equiv 1 \quad \text{in } \tilde{I}_m,\quad 0\leq \zeta \leq 1 \quad\text{and}\quad |\partial_t\zeta|\leq \frac{2}{\ell_m^2-\tilde{\ell}_m^2}\leq \frac{2}{(\ell_m-\tilde{\ell}_m)^2}.
    $$
    Using Lemma \ref{lem : Caccioppoli inequality} with $k=k_{m+1}^\pm$ gives
    \begin{align}
        &\sup_{t\in I_m} \dashint_{B_m} \frac{(u-k_{m+1}^\pm)_\pm^2\eta^q\zeta^2}{\ell_m^2}\, dx + \miint{Q_m} \eta^qH(z,|D(u-k_{m+1}^\pm)_\pm|)\zeta^2\, dz\nonumber\\
        &\qquad \leq \frac{2^{mp}c}{(1-\sigma)^p R^p}\miint{Q_m} (u-k_{m+1}^\pm)_\pm^p\, dz + \frac{2^{mq}c}{(1-\sigma)^q R^q}\miint{Q_m} a(z)(u-k_{m+1}^\pm)^q_\pm\, dz\nonumber\\\label{eq : Caccioppoli inequality with level k}
        &\qquad \quad +\frac{2^{2m}c}{(1-\sigma)^2\ell^2}\miint{Q_m} (u-k_{m+1}^\pm)_\pm^2\, dz
    \end{align}
    for some $c=c(n,p,q,\nu,L)>0$. First, we assume that $k$ is positive. Since $k_m^+\leq k_{m+1}^+$, we have
    $$
    \{(u-k_{m+1}^+)_+>0\}\subset\{(u-k_m^+)_+>0\}.
    $$
    It implies that
    \begin{align}
        \iints{Q_m} (u-k_m^+)_+^{\tilde{p}}\, dz &\geq \iints{Q_m} (u-k_m^+)_+^{\tilde{p}} \chi_{\{u>k_{m+1}^+\}}(z)\, dz\nonumber\\
        &\geq (k_{m+1}^+-k_m^+)^{\tilde{p}}|A_{m+1}^+|\nonumber\\
        &=\frac{|k|^{\tilde{p}}}{2^{(m+1)\tilde{p}}}|A_{m+1}^+| \nonumber
    \end{align}
    and hence
    \begin{equation}\label{eq : estimate of level set on +}
        \begin{aligned}
        \iints{Q_m} (u-k_{m+1}^+)_+^s\, dz &\leq \left(\iints{Q_m} (u-k_{m+1}^+)_+^{\tilde{p}}\, dz\right)^{\frac{s}{\tilde{p}}}|A_{m+1}^+|^{1-\frac{s}{\tilde{p}}}\\
        &\leq \frac{2^{(m+1)(\tilde{p}-s)}}{|k|^{\tilde{p}-s}}\iints{Q_m} (u-k_{m}^+)_+^{\tilde{p}}\, dz
        \end{aligned}
    \end{equation}
    for $s\in\{2,p,q\}$. In the following, we assume $|k|\geq 1$, which will be guaranteed by the choice of $|k|$ in \eqref{eq : choice of level k} below. Recalling $d=\max\{q,2\}$, we have $p<q\leq d$ and $2\leq d<\tilde{p}$, so that
    $$
    \frac{1}{(1-\sigma)^s}\leq \frac{1}{(1-\sigma)^d}\quad\text{and}\quad \frac{1}{|k|^{\tilde{p}-s}}\leq \frac{1}{|k|^{\tilde{p}-d}}\qquad \text{for } s\in\{2,p,q\}.
    $$
    The first inequality is used in \eqref{eq : Caccioppoli inequality with level k on +} and \eqref{eq : Caccioppoli inequality with level k on -} below, while the second one is used in the display preceding \eqref{eq : final recursive inequality}. Using \eqref{eq : estimate of level set on +} in \eqref{eq : Caccioppoli inequality with level k}, we have
    \begin{align}
        &\sup_{t\in I_m} \dashint_{B_m} \frac{(u-k_{m+1}^+)_+^2\eta^q\zeta^2}{\ell_m^2}\, dx + \miint{Q_m} \eta^qH(z,|D(u-k_{m+1}^+)_+|)\zeta^2\, dz\nonumber\\\label{eq : Caccioppoli inequality with level k on +}
        &\qquad \leq \frac{2^{\tilde{p} m}c}{(1-\sigma)^{d}}\left(\frac{R^{\tilde{p}-p}}{|k|^{\tilde{p}-p}}+\frac{R^{\tilde{p}-q}}{|k|^{\tilde{p}-q}}+\frac{R^{\tilde{p}}}{\ell^2 |k|^{\tilde{p}-2}}\right)\miint{Q_m}\left(\frac{(u-k^+_m)_+}{R}\right)^{\tilde{p}}\, dz
    \end{align}
    for some $c=c(n,p,q,\nu,L,\|a\|_{L^\infty(\Omega_T)})>1$. Now, we assume that $k$ is negative. Since $k_m^-\geq k_{m+1}^-$, we have
    $$
    \{(u-k_{m+1}^-)_->0\}\subset \{(u-k_m^-)_->0\}.
    $$
    Thus,
    \begin{align}
        \iints{Q_m} (u-k_m^-)_-^{\tilde{p}}\, dz &\geq \iints{Q_m} (u-k_m^-)_-^{\tilde{p}} \chi_{\{u<k_{m+1}^-\}}(z)\, dz\nonumber\\
        &\geq (k_m^- - k_{m+1}^-)^{\tilde{p}}|A_{m+1}^-|\nonumber\\
        &=\frac{|k|^{\tilde{p}}}{2^{(m+1){\tilde{p}}}}|A_{m+1}^-|\nonumber
    \end{align}
    and hence
    \begin{equation}\label{eq : estimate of level set on -}
        \begin{aligned}
        \iints{Q_m} (u-k_{m+1}^-)_-^s\, dz &\leq \left(\iints{Q_m} (u-k_{m+1}^-)_-^{\tilde{p}}\, dz\right)^{\frac{s}{\tilde{p}}}|A_{m+1}^-|^{1-\frac{s}{\tilde{p}}}\\
        &\leq \frac{2^{(m+1)(\tilde{p}-s)}}{|k|^{\tilde{p}-s}}\iints{Q_m} (u-k_{m}^-)_-^{\tilde{p}}\, dz
        \end{aligned}
    \end{equation}
    Using \eqref{eq : estimate of level set on -} in \eqref{eq : Caccioppoli inequality with level k}, we have
    \begin{align}
        &\sup_{t\in I_m} \dashint_{B_m} \frac{(u-k_{m+1}^-)_-^2\eta^q\zeta^2}{\ell_m^2}\, dx + \miint{Q_m} \eta^qH(z,|D(u-k_{m+1}^-)_-|)\zeta^2\, dz\nonumber\\\label{eq : Caccioppoli inequality with level k on -}
        &\qquad \leq \frac{2^{\tilde{p} m}c}{(1-\sigma)^{d}}\left(\frac{R^{\tilde{p}-p}}{|k|^{\tilde{p}-p}}+\frac{R^{\tilde{p}-q}}{|k|^{\tilde{p}-q}}+\frac{R^{\tilde{p}}}{\ell^2 |k|^{\tilde{p}-2}}\right)\miint{Q_m}\left(\frac{(u-k^-_m)_-}{R}\right)^{\tilde{p}}\, dz
    \end{align}
    for some $c=c(n,p,q,\nu,L,\|a\|_{L^\infty(\Omega_T)})>1$. By \eqref{eq : a parabolic embedding inequality} with $f=(u-k_{m+1}^\pm)_\pm$, $\mathcal{Q}=\tilde{Q}_m$, $\hat{\mathcal{Q}}=Q_{m+1}$ and $|D\iota|\leq \frac{2^{m+2}}{(1-\sigma)R}$ and dividing both sides by $|Q_{m+1}|$, we have
    \begin{equation}\label{eq : parabolic embedding ienq in main theorem proof}
    \begin{aligned}
        &\miint{Q_{m+1}}\left(\frac{(u-k_{m+1}^\pm)_{\pm}}{\tilde{R}_m}\right)^{\tilde{p}}\, dz\\ 
        &\qquad\leq c\left(\miint{\tilde{Q}_m}\left[|D(u-k_{m+1}^\pm)_\pm|^p+\frac{2^{(m+2)p}(u-k_{m+1}^\pm)_\pm^p}{(1-\sigma)^pR^p}\right]\,dz\right)\\
        &\qquad\qquad \times \left(\sup_{t\in\tilde{I}_m}\dashint_{\tilde{B}_m}\frac{(u-k_{m+1}^\pm)_\pm^2}{\tilde{R}_m^2}\,dx\right)^{\frac{p}{n}}.
    \end{aligned}
\end{equation}
    Using \eqref{eq : estimate of level set on +}, \eqref{eq : estimate of level set on -}, \eqref{eq : Caccioppoli inequality with level k on +} and \eqref{eq : Caccioppoli inequality with level k on -} and recalling the definitions of $\eta$ and $\zeta$, we obtain
    $$
    \begin{aligned}
        &\miint{Q_{m+1}}\left(\frac{(u-k_{m+1}^\pm)_{\pm}}{\tilde{R}_m}\right)^{\tilde{p}}\, dz\\ 
        &\;\leq c \left(\frac{2^{\tilde{p} m}}{(1-\sigma)^d}\left(\frac{R^{\tilde{p}-p}}{|k|^{\tilde{p}-p}}+\frac{R^{\tilde{p}-q}}{|k|^{\tilde{p}-q}}+\frac{R^{\tilde{p}}}{\ell^2 |k|^{\tilde{p}-2}}\right)\miint{Q_m}\left(\frac{(u-k_m^\pm)_\pm}{R}\right)^{\tilde{p}}\, dz\right)\\
        &\;\; \times \left(\frac{\tilde{\ell}_m^2}{\tilde{R}_m^2}\frac{2^{\tilde{p}m}}{(1-\sigma)^d}\left(\frac{R^{\tilde{p}-p}}{|k|^{\tilde{p}-p}}+\frac{R^{\tilde{p}-q}}{|k|^{\tilde{p}-q}}+\frac{R^{\tilde{p}}}{\ell^2 |k|^{\tilde{p}-2}}\right)\miint{Q_m}\left(\frac{(u-k_m^\pm)_\pm}{R}\right)^{\tilde{p}}\, dz\right)^\frac{p}{n}.
    \end{aligned}
    $$
    Now, for any $\rho\in(0,1)$, take $R=\frac{\rho}{2}$ and $\ell = \left(\frac{\rho}{2}\right)^{\frac{\tilde{p}}{2}}$, so that $Q_0=\mathcal{Q}_{\frac{\rho}{2},\left(\frac{\rho}{2}\right)^{\tilde{p}}}(z_0)\subset \mathcal{Q}_{\rho,\rho^{\tilde{p}}}(z_0)\subset\subset\Omega_T$. Since $\tilde{R}_m\leq R=\frac{\rho}{2}<1$, $\frac{R^{\tilde{p}}}{\ell^2}=1$ and $\frac{\tilde{\ell}_m^2}{\tilde{R}_m^2}=\frac{\ell^2}{R^2}=\left(\frac{\rho}{2}\right)^{\tilde{p}-2}\leq 1$ (as $2<\tilde{p}$ and $\rho<1$), we have
    $$
    \begin{aligned}
        &\miint{Q_{m+1}}\left(\frac{(u-k_{m+1}^\pm)_{\pm}}{\rho}\right)^{\tilde{p}}\, dz\\ 
        &\quad\leq c \left(\frac{2^{\tilde{p} m}}{(1-\sigma)^{d}}\left(\frac{1}{|k|^{\tilde{p}-p}}+\frac{1}{|k|^{\tilde{p}-q}}+\frac{1}{|k|^{\tilde{p}-2}}\right)\miint{Q_m}\left(\frac{(u-k_m^\pm)_\pm}{\rho}\right)^{\tilde{p}}\, dz\right)\\
        &\quad\quad \times \left(\frac{2^{\tilde{p}m}}{(1-\sigma)^{d}}\left(\frac{1}{|k|^{\tilde{p}-p}}+\frac{1}{|k|^{\tilde{p}-q}}+\frac{1}{|k|^{\tilde{p}-2}}\right)\miint{Q_m}\left(\frac{(u-k_m^\pm)_\pm}{\rho}\right)^{\tilde{p}}\, dz\right)^\frac{p}{n}\\
        &\quad\leq c\left(\frac{2^{\tilde{p}m}|k|^{d-\tilde{p}}}{(1-\sigma)^{d}}\miint{Q_m}\left(\frac{(u-k_m^\pm)_\pm}{\rho}\right)^{\tilde{p}}\, dz\right)^{1+\frac{p}{n}},
    \end{aligned}
    $$
    where we used $|k|\geq 1$ and $p<q\leq d$, $2\leq d$ in the last step. Now, write $\vartheta=\frac{p}{n}>0$ and consider
    $$Y_m=\miint{Q_m}\left(\frac{(u-k_{m}^\pm)_\pm}{\rho}\right)^{\tilde{p}}\, dz.
    $$
    Then
    \begin{equation}\label{eq : final recursive inequality}
    Y_{m+1}\leq c\left(\frac{2^{\tilde{p} m}|k|^{d-\tilde{p}}}{(1-\sigma)^d}\right)^{1+\vartheta} Y_m^{1+\vartheta}.
    \end{equation}
    The following is a version of the standard fast geometric convergence argument, see for instance \cite[Chapter I, Lemma 4.1]{1993_Degenerate_parabolic_equations_DiBenedetto}; we include the short proof to keep the paper self-contained. Now, we prove inductively for any $m\in \mathbb{N}\cup\{0\}$,
    \begin{equation}\label{eq : estimate of Y_n}
    Y_m\leq \frac{Y_0}{\lambda^m}
    \end{equation}
    for some $\lambda>1$, for the choice of $|k|$ in \eqref{eq : choice of level k} below. Clearly, it is true for $m=0$. Suppose 
    \begin{equation*}
        Y_{m-1}\leq \frac{Y_0}{\lambda^{m-1}}.
    \end{equation*}
    Then we have
    $$
    Y_{m-1}^{1+\vartheta}\leq \frac{Y_0^\vartheta}{\lambda^{m\vartheta-(1+\vartheta)}}\frac{Y_0}{\lambda^m}.
    $$
    Taking $\lambda$ satisfying
    $$
    \lambda^\vartheta=2^{\tilde{p}(1+\vartheta)}, 
    $$
    we obtain, with a larger constant $c=c(n,p,q,\nu,L,\|a\|_{L^\infty(\Omega_T)})$ which now depends also on $\lambda=\lambda(n,p)$,
    $$
    Y_m\leq c \left(\frac{|k|^{d-\tilde{p}}}{(1-\sigma)^{d}}\right)^{1+\vartheta}Y_0^{\vartheta}\frac{Y_0}{\lambda^m}.
    $$
    Note that 
    $$Y_0=\miint{Q_0}\left(\frac{u_\pm}{\rho}\right)^{\tilde{p}}\, dz
    $$ 
    does not depend on $k$, since $k_0^\pm=0$. If $Y_0=0$, then $u_\pm=0$ a.e. in $Q_0$ and there is nothing to prove; hence we may assume $Y_0>0$. Next, we take
    \begin{equation}\label{eq : choice of level k}
    |k|=1+c_\star(1-\sigma)^{-\frac{d}{\tilde{p}-d}}Y_0^{\frac{\vartheta}{(\tilde{p}-d)(1+\vartheta)}},
    \end{equation}
    which in particular guarantees $|k|\geq 1$. Since $d<\tilde{p}$, we have
    $$
    |k|^{d-\tilde{p}}\leq \left(c_\star(1-\sigma)^{-\frac{d}{\tilde{p}-d}}Y_0^{\frac{\vartheta}{(\tilde{p}-d)(1+\vartheta)}}\right)^{d-\tilde{p}}=c_\star^{d-\tilde{p}}(1-\sigma)^{d}Y_0^{-\frac{\vartheta}{1+\vartheta}}
    $$
    and hence
    $$
    c \left(\frac{|k|^{d-\tilde{p}}}{(1-\sigma)^d}\right)^{1+\vartheta}Y_0^{\vartheta}\leq c\, c_\star^{(d-\tilde{p})(1+\vartheta)}.
    $$
    Here, we choose $c_\star=c_\star({n,p,q,\nu,L,\|a\|_{L^\infty(\Omega_T)}})>1$ sufficiently large so that
    $$
    c \, c_\star^{(d-\tilde{p})(1+\vartheta)}\leq 1.
    $$
    Thus, \eqref{eq : estimate of Y_n} holds. Letting $m\rightarrow\infty$ in \eqref{eq : estimate of Y_n}, we get
    $$
    Y_m\rightarrow 0
    $$
    and hence
    $$
    \sup_{\mathcal{Q}_{\sigma\frac{\rho}{2},\sigma^2 \left(\frac{\rho}{2}\right)^{\tilde{p}}}(z_0)} |u|\leq 1+c_\star(1-\sigma)^{-\frac{d}{\tilde{p}-d}}Y_0^{\frac{\vartheta}{(\tilde{p}-d)(1+\vartheta)}}.
    $$
    Furthermore, applying \eqref{eq : a parabolic embedding inequality} in a similar manner to \eqref{eq : parabolic embedding ienq in main theorem proof}, we obtain
    $$
    \begin{aligned}
        Y_0&=\miint{\mathcal{Q}_{\frac{\rho}{2},\left(\frac{\rho}{2}\right)^{\tilde{p}}}(z_0)} \left(\frac{u_\pm}{\rho}\right)^{\tilde{p}}\, dz\leq \miint{\mathcal{Q}_{\frac{\rho}{2},\left(\frac{\rho}{2}\right)^{\tilde{p}}}(z_0)} \left(\frac{|u|}{\rho}\right)^{\tilde{p}}\, dz\\
        &\leq c\left(\miint{\mathcal{Q}_{\rho,\rho^{\tilde{p}}}(z_0)}\left[|Du|^p+\frac{|u|^p}{\rho^p}\right]\, dz\right)\left(\sup_{t\in\mathcal{I}_{\rho^{\tilde{p}}}(t_0)}\dashint_{B_{\rho}(x_0)}\frac{|u|^2}{\rho^2}\,dx\right)^{\frac{p}{n}}
    \end{aligned}
    $$
    for some $c=c(n,p)>1$. Combining the last two displays completes the proof.
\end{proof}


\begin{thebibliography}{10}

\bibitem{Baroni2018}
P.~Baroni, M.~Colombo, and G.~Mingione.
\newblock Regularity for general functionals with double phase.
\newblock {\em Calc. Var. Partial Differential Equations}, 57(2):Paper No. 62,
  48, 2018.

\bibitem{Buryachenko2022}
K.~O. Buryachenko and I.~I. Skrypnik.
\newblock Local continuity and {H}arnack's inequality for double-phase
  parabolic equations.
\newblock {\em Potential Anal.}, 56(1):137--164, 2022.

\bibitem{Chlebicks2019}
I.~Chlebicka, P.~Gwiazda, and A.~Zatorska-Goldstein.
\newblock Parabolic equation in time and space dependent anisotropic
  {M}usielak-{O}rlicz spaces in absence of {L}avrentiev's phenomenon.
\newblock {\em Ann. Inst. H. Poincar\'{e} C Anal. Non Lin\'{e}aire},
  36(5):1431--1465, 2019.

\bibitem{Colombo2015}
M.~Colombo and G.~Mingione.
\newblock Regularity for double phase variational problems.
\newblock {\em Arch. Ration. Mech. Anal.}, 215(2):443--496, 2015.

\bibitem{1993_Degenerate_parabolic_equations_DiBenedetto}
E.~DiBenedetto.
\newblock {\em Degenerate parabolic equations}.
\newblock Universitext. Springer-Verlag, New York, 1993.

\bibitem{Esposito2004}
L.~Esposito, F.~Leonetti, and G.~Mingione.
\newblock Sharp regularity for functionals with {$(p,q)$} growth.
\newblock {\em J. Differential Equations}, 204(1):5--55, 2004.

\bibitem{kim2026absencelavrentievphenomenondegenerate}
B.~Kim, Y.~Kim, and J.~Oh.
\newblock Absence of the {L}avrentiev phenomenon for degenerate parabolic double
  phase problems.
\newblock arXiv preprint arXiv:2603.14235, 2026.

\bibitem{Kim2024}
B.~Kim and J.~Oh.
\newblock Higher integrability for weak solutions to parabolic multi-phase
  equations.
\newblock {\em J. Differential Equations}, 409:223--298, 2024.

\bibitem{kim2025boundedsolutionsinterpolativegap}
B.~Kim and J.~Oh.
\newblock Bounded solutions and interpolative gap bounds for degenerate
  parabolic double phase problems.
\newblock arXiv preprint arXiv:2511.13454, 2025.

\bibitem{kim2026interpolativerefinementgapbound}
B.~Kim and J.~Oh.
\newblock Interpolative refinement of gap bound conditions for singular
  parabolic double phase problems.
\newblock arXiv preprint arXiv:2601.01571, 2026.

\bibitem{Kim2025}
B.~Kim, J.~Oh, and A.~Sen.
\newblock Parabolic {L}ipschitz truncation for multi-phase problems: the
  degenerate case.
\newblock {\em Adv. Calc. Var.}, 18(3):979--1010, 2025.

\bibitem{Wontae2023b}
W.~Kim.
\newblock Calder\'on-{Z}ygmund type estimate for the singular parabolic
  double-phase system.
\newblock {\em J. Math. Anal. Appl.}, 551(1):Paper No. 129593, 33, 2025.

\bibitem{2023_Gradient_Higher_Integrability_for_Degenerate_Parabolic_Double-Phase_Systems}
W.~Kim, J.~Kinnunen, and K.~Moring.
\newblock Gradient higher integrability for degenerate parabolic double-phase
  systems.
\newblock {\em Arch. Ration. Mech. Anal.}, 247(5):Paper No. 79, 46, 2023.

\bibitem{Wontae2023a}
W.~Kim, J.~Kinnunen, and L.~S\"arki\"o.
\newblock Lipschitz truncation method for parabolic double-phase systems and
  applications.
\newblock {\em J. Funct. Anal.}, 288(3):Paper No. 110738, 60, 2025.

\bibitem{Wontae2025}
W.~Kim, K.~Moring, and L.~S\"arki\"o.
\newblock H\"older regularity for degenerate parabolic double-phase equations.
\newblock {\em J. Differential Equations}, 434:Paper No. 113231, 34, 2025.

\bibitem{Wontae2024}
W.~Kim and L.~S\"arki\"o.
\newblock Gradient higher integrability for singular parabolic double-phase
  systems.
\newblock {\em NoDEA Nonlinear Differential Equations Appl.}, 31(3):Paper No.
  40, 38, 2024.

\bibitem{Ok2017}
J.~Ok.
\newblock Regularity of {$\omega$}-minimizers for a class of functionals with
  non-standard growth.
\newblock {\em Calc. Var. Partial Differential Equations}, 56(2):Paper No. 48,
  31, 2017.

\bibitem{Sen2025}
A.~Sen.
\newblock Gradient higher integrability for degenerate/singular parabolic
  multi-phase problems.
\newblock {\em J. Geom. Anal.}, 35(6):Paper No. 170, 95, 2025.

\bibitem{Singer2016}
T.~Singer.
\newblock Existence of weak solutions of parabolic systems with {$p,
  q$}-growth.
\newblock {\em Manuscripta Math.}, 151(1-2):87--112, 2016.

\bibitem{Singer2016NoDEA}
T.~Singer.
\newblock Local boundedness of variational solutions to evolutionary problems
  with non-standard growth.
\newblock {\em NoDEA Nonlinear Differential Equations Appl.}, 23(2):Paper No.
  19, 23, 2016.

\bibitem{Zhikov1986}
V.~V. Zhikov.
\newblock Averaging of functionals of the calculus of variations and elasticity
  theory.
\newblock {\em Izv. Akad. Nauk SSSR Ser. Mat.}, 50(4):675--710, 877, 1986.

\bibitem{Zhikov1993}
V.~V. Zhikov.
\newblock Lavrentiev phenomenon and homogenization for some variational
  problems.
\newblock {\em C. R. Acad. Sci. Paris S\'{e}r. I Math.}, 316(5):435--439, 1993.

\bibitem{Zhikov1995}
V.~V. Zhikov.
\newblock On {L}avrentiev's phenomenon.
\newblock {\em Russian J. Math. Phys.}, 3(2):249--269, 1995.

\bibitem{Zhikov1997}
V.~V. Zhikov.
\newblock On some variational problems.
\newblock {\em Russian J. Math. Phys.}, 5(1):105--116 (1998), 1997.

\end{thebibliography}
\end{document}